\documentclass[11pt,twoside]{article}
\usepackage{mathrsfs}
\usepackage{amsmath}
\usepackage{amsthm}
\input amssymb.sty
\usepackage{amsfonts}
\usepackage{amssymb}
\usepackage{latexsym}
\usepackage[all]{xy}

\date{\empty}
\allowdisplaybreaks[4] \footskip=15pt

\numberwithin{equation}{section} \theoremstyle{plain}
\newtheorem*{thm*}{Main Theorem}
\newtheorem{theorem}{Theorem}[section]
\newtheorem{corollary}[theorem]{Corollary}
\newtheorem*{corollary*}{Corollary}

\newtheorem*{claim*}{Claim}
\newtheorem{lemma}[theorem]{Lemma}
\newtheorem*{lemma*}{Lemma}
\newtheorem{proposition}[theorem]{Proposition}
\newtheorem*{proposition*}{Proposition}
\newtheorem{remark}[theorem]{Remark}
\newtheorem*{remark*}{Remark}
\newtheorem{example}[theorem]{Example}
\newtheorem*{example*}{Example}

\newtheorem*{question*}{Question}

\newtheorem*{definition*}{Definition}

\newtheorem*{acknowledgements*}{ACKNOWLEDGEMENTS}

\begin{document}
\begin{center}
{\large  \bf Centralizer's applications to the inverse along an element}\\
\vspace{0.8cm} {\small \bf   Huihui Zhu$^{[1,2]}$, Jianlong
Chen$^{[1]}$\footnote{Corresponding author

1 Department of Mathematics, Southeast University, Nanjing 210096,
China.

2 CMAT-Centro de Matem\'{a}tica, Universidade do Minho, Braga
4710-057, Portugal.

3 Departamento de Matem\'{a}tica e Aplica\c{c}\~{o}es, Universidade
do Minho, Braga 4710-057, Portugal.

4 Universit\'{e} Paris-Ouest Nanterre-La D\'{e}fense, Laboratoire
Modal'X, 200 avenue de la r\'{e}publique, 92000 Nanterre, France.

Email: ahzhh08@sina.com(H. Zhu), jlchen@seu.edu.cn(J. Chen),
pedro@math.uminho.pt (P. Patr\'{i}cio), xavier.mary@u-paris10.fr (X.
Mary).}, Pedro Patr\'{i}cio$^{[2,3]}$, Xavier Mary$^{[4]}$}
\end{center}

\bigskip

{\bf  Abstract:}  \leftskip0truemm\rightskip0truemm  In this paper,
we first prove that the absorption law for one-sided inverses along
an element holds, deriving the absorption law for the inverse along
an element. We then apply this result to obtain the absorption law
for the inverse along different elements. Also, the reverse order
law and the existence criterion for the inverse along an element are
given by centralizers in a ring. Finally, we characterize the
Moore-Penrose inverse by one-sided invertibilities in a ring with
involution.

\textbf{Keywords:} Absorption laws, Reverse order laws, Centralizers
, Left (Right) inverses along an element, The inverse along an
element

\textbf{AMS Subject Classifications:} 15A09, 16U80, 16W10
\bigskip

%%%%%%%%%%%%%%%%%%%%%%%%%%%%%%%%%%%%%%%%%%%%%%%%%%%%%%%%%%%%%%%%%%%%%
%%%%%%%%%%%%%%%%%%%%%    Section 1   %%%%%%%%%%%%%%%%%%%%%%%%%%%%
%%%%%%%%%%%%%%%%%%%%%%%%%%%%%%%%%%%%%%%%%%%%%%%%%%%%%%%%%%%%%%%%%%%%%

%% \linenumbers

%% main text
\section { \bf Introduction}

It is well-known that $a^{-1}+b^{-1}=a^{-1}(a+b)b^{-1}$ for any
invertible elements $a$ and $b$ in a ring. The equality above is
known as the absorption law. In general, the absorption laws for
group inverses, Drazin inverses, Moore-Penrose inverses,
\{1,3\}-inverses and \{1,4\}-inverses do not hold. So, many papers
\cite{Chen,Jin,Liu} devoted to the study of these aspects.

Recently, authors \cite{Zhu and chen} introduced a new type of
generalized inverse called one-sided inverses along an element,
which can been seen as a generalization of group inverses, Drazin
inverses, Moore-Penrose inverses and the inverse along an element.
It is natural to consider whether the absorption law for such
inverses holds.

In this paper, we first prove that the absorption law for one-sided
inverses along an element holds in a ring. As applications, the
absorption law for the inverse along the same element, i.e.,
$a^{\parallel d}+b^{\parallel d}=a^{\parallel d}(a+b)b^{\parallel
d}$, is obtained. We then apply this result to obtain the absorption
law for the inverse along different elements. Also, the reverse
order law for the inverse along an element is considered.
Furthermore, we derive an existence criterion of the inverse along
an element by centralizers in a ring. Finally, we characterize the
Moore-Penrose inverse in terms of one-sided invertibilities,
extending the results in \cite{Patricio and Mendes,Zhu reverse}.

Let us now recall some notions of generalized inverses. We say that
$a \in R$ is (von Neumann) regular if there exists $x$ in $R$ such
that $a=axa$. Such $x$ is called an inner inverse or \{1\}-inverse
of $a$, and is denoted by $a^{-}$.

Following \cite{Drazin}, an element $a\in R$ is said to be Drazin
invertible if there exist $b\in R$ and positive integer $k$ such
that the following conditions hold:
\begin{center}
(i) $ab$ = $ba$, (ii) $b^2a$ = $b$, (iii) $a^k$ = $a^{k+1}b$.
\end{center}
The element $b$ satisfying the above conditions (i)-(iii) is unique
if it exists, and is denoted by $a^D$. The smallest positive integer
$k$ in condition (iii) is called the Drazin index of $a$ and is
denoted by ind$(a)$. We call $a$ group invertible if $a$ is Drazin
invertible with ind$(a)$ = 1.

Let $*$ be an involution on $R$, that is the involution $*$
satisfies $(x^*)^* = x$, $(xy)^* = y^*x^*$ and $(x+y)^*=x^*+y^*$ for
all $x,y\in R$. An element $a\in R$ (with involution $*$) is
Moore-Penrose invertible \cite{Penrose} if there exists $b\in R$
satisfying the following equations
\begin{center}
${\rm(i)}~aba=a$, ${\rm (ii)}~bab=b$, ${\rm (iii)}~(ab)^*=ab$, ${\rm
(iv)}~(ba)^*=ba$.
\end{center}
Any element $b$ satisfying the equations above is called a
Moore-Penrose inverse of $a$. If such $b$ exists, then it is unique
and is denoted by $a^\dag$. If $x$ satisfies the equations (i) and
(iii), then $x$ is called a $\{1,3\}$-inverse of $a$, and is denoted
by $a^{(1,3)}$. If $x$ satisfies the equations (i) and (iv), then
$x$ is called a $\{1,4\}$-inverse of $a$, and is denoted by
$a^{(1,4)}$.

Throughout this paper, we assume that $R$ is an associative ring
with unity 1. Let $a,b,d\in R$. An element $b$ is called a left
(resp., right) inverse of $a$ along $d$ \cite{Zhu and chen} if
$bad=d$ (resp., $dab=b$) and $b\in Rd$ (resp., $b \in dR$). By
$a_l^{\parallel d}$ and $a_r^{\parallel d}$ we denote a left and a
right inverse of $a$ along $d$, respectively.  Furthermore, an
element $a$ is called invertible along $d$ \cite{Mary} if there
exists $b$ such that $bad=d=dab$ and $b\in d R \cap R d$. Such $b$
is unique if it exists, and is denoted by $a^{\parallel d}$.  It is
known \cite{Zhu and chen} that $a$ is both left and right invertible
along $d$ if and only if it is invertible along $d$. More results on
(one-sided) inverses along an element can be referred to \cite{Zhu
and chen,Zhu reverse,Zhu and p}.

\section{Absorption laws for the inverse along an element}

The main goal of this section is to illustrate that the absorption
laws for (left, right) inverses along an element hold in a ring. We
first begin with the following lemma.

\begin{lemma} \label{absorption law} Let $a,b,d\in R$. Then

\emph{(i)} If $a_l^{\parallel d}$ and $b_r^{\parallel d}$ exist,
then $a_l^{\parallel d}ab_r^{\parallel d}=b_r^{\parallel d}$ and
$a_l^{\parallel d}bb_r^{\parallel d}=a_l^{\parallel d}$.

\emph{(ii)} If $a_r^{\parallel d}$ and $b_l^{\parallel d}$ exist,
then $b_l^{\parallel d}ba_r^{\parallel d}=a_r^{\parallel d}$ and
$b_l^{\parallel d}aa_r^{\parallel d}=b_l^{\parallel d}$.
\end{lemma}

\begin{proof}  (i) Note that $a_l^{\parallel d}$ can be written as the form $xd$ for some $x\in R$. Also, there exists $y\in R$ such that $b_r^{\parallel d}=dy$.
Hence, it follows that $a_l^{\parallel d}ab_r^{\parallel d}=a_l^{\parallel d}ady=dy=b_r^{\parallel d}$ and $a_l^{\parallel d}bb_r^{\parallel d}=xdbb_r^{\parallel d}
=a_l^{\parallel d}$.

(ii) By the symmetry of $a$ and $b$.
\end{proof}

Applying Lemma \ref{absorption law}, we get the absorption laws for
one-sided inverses along an element.

\begin{theorem} \label{al} Let $a,b,d\in R$. Then

\emph{(i)} If $a_l^{\parallel d}$ and $b_r^{\parallel d}$ exist,
then $a_l^{\parallel d}+b_r^{\parallel d}=a_l^{\parallel
d}(a+b)b_r^{\parallel d}$.

\emph{(ii)} If $a_r^{\parallel d}$ and $b_l^{\parallel d}$ exist,
then $a_r^{\parallel d}+b_l^{\parallel d}=b_l^{\parallel
d}(a+b)a_r^{\parallel d}$.
\end{theorem}

\begin{proof} (i) It follows from Lemma \ref{absorption law}(i) that
$$a_l^{\parallel d}(a+b)b_r^{\parallel d}=a_l^{\parallel d}ab_r^{\parallel d}+a_l^{\parallel d}bb_r^{\parallel d}=a_l^{\parallel d}+b_r^{\parallel d}.$$

(ii) It is a direct check by Lemma \ref{absorption law}(ii).
\end{proof}

As a corollary of Theorem \ref{al}, we obtain the absorption law for
the inverse along an element.

\begin{corollary} \label{all} Let $a,b,d\in R$ and let $a^{\parallel d}$ and $b^{\parallel d}$ exist. Then
$a^{\parallel d}+b^{\parallel d}=a^{\parallel d}(a+b)b^{\parallel d}$.
\end{corollary}

From Corollary \ref{all}, we know that the absorption law  for the
inverse along the same element holds in a ring. It is natural and
interesting to consider whether the absorption law for the inverse
along different elements, i.e., $a^{\parallel d_1}+b^{\parallel
d_2}=a^{\parallel d_1}(a+b)b^{\parallel d_2}$ holds? In general,
$a^{\parallel d_1}+b^{\parallel d_2} = a^{\parallel
d_1}(a+b)b^{\parallel d_2}$ does not hold as the following example
shows.

\begin{example}
{\rm Let $M_2(\mathbb{C})$ be the ring of all $2 \times 2$ complex
matrices. Take $A=\left[\begin{smallmatrix}
              1   & 0 \\
              1 & 0
        \end{smallmatrix}
 \right]$, $B=\left[\begin{smallmatrix}
              0   & 0 \\
              1 & 1
        \end{smallmatrix}
 \right]$, $D_1=\left[\begin{smallmatrix}
              1   & 1 \\
              0 & 0
        \end{smallmatrix}
 \right]$ and $D_2=\left[\begin{smallmatrix}
              1   & 1 \\
              1 & 1
        \end{smallmatrix}
 \right]$ in $M_2(\mathbb{C})$. We get $A^{\parallel D_1}=
 \left[\begin{smallmatrix}
              \frac{1}{2}   & \frac{1}{2} \\
              0 & 0
        \end{smallmatrix}
 \right]$ and $B^{\parallel D_2}=
 \left[\begin{smallmatrix}
              \frac{1}{2}   & \frac{1}{2} \\
              \frac{1}{2} & \frac{1}{2}
        \end{smallmatrix}
 \right]$ by direct calculation. However, $A^{\parallel D_1}+B^{\parallel D_2}=\left[\begin{smallmatrix}
              1   & 1 \\
              \frac{1}{2} & \frac{1}{2}
        \end{smallmatrix}
 \right] \neq \left[\begin{smallmatrix}
              1   & 1 \\
              0 & 0
        \end{smallmatrix}
 \right]=A^{\parallel D_1}(A+B)B^{\parallel D_2}$.}
\end{example}

Next, we consider under what conditions, the absorption law for the
inverse along different elements to hold. First, we recall the
notion of centralizers.

In 1952, Wendel \cite{Wendel} introduced the notion of left
centralizers in group algebras. Then, Johnson \cite{Johnson} gave an
introduction to the theory of left and right centralizers in
semigroups. In 1966, Kellogg \cite{Kellogg} further investigated
centralizers in $H^*$-algebra. Later, Vukman
\cite{Vukman,Vukman2005} and Zalar \cite{Zalar} studied centralizers
in operator algebras and semiprime rings, respectively. More
recently, see \cite{Zhu centralizer}, a map $\sigma$ in a semigroup
$S$ is called a left (resp., right) centralizer if $\sigma(ab)$ =
$\sigma(a)b$ (resp., $\sigma(ab)$ = $a\sigma(b)$) for all $a, b\in
S$. We call $\sigma$ a centralizer if it is both a left and a right
centralizer, i.e., $a\sigma(b)$ = $\sigma(ab)$ = $\sigma(a)b$ for
all $a, b\in S$.

Herein, we remind the reader some examples of left centralizers,
centralizers and bijective centralizers.

\begin{example} \label{Ex} Let $a,x\in R$ and let $\sigma: R\rightarrow R, \sigma(a)= xa$.
Then

\emph{(i)} The map $\sigma: a\mapsto xa$ is a left centralizer.

\emph{(ii)}  The map $\sigma: a\mapsto xa$ is a centralizer if $x$
is a central element.

\emph{(iii)} The map $\sigma: a\mapsto xa$ is a centralizer if $x$
is an invertible central element. Moreover, if $\sigma$ is a
bijective centralizer then so is $\sigma^{-1}$.
\end{example}

\begin{theorem} \label{lambda} Let $\sigma: R \rightarrow R$ be a bijective centralizer and let $a,b,d_1,d_2\in R$ with $d_1=\sigma (d_2)$.
If $a^{\parallel d_1}$ and $b^{\parallel d_2}$ exist, then $a^{\parallel d_1}+b^{\parallel d_2}=a^{\parallel d_1}(a+b)b^{\parallel d_2}$.
\end{theorem}

\begin{proof} It is known \cite{Mary and Patricio} that the existence of $b^{\parallel d_2}$ implies that $d_2$ is regular. Hence
$d_1=\sigma(d_2)=\sigma(d_2d_2^{-}d_2)=d_2d_2^{-}\sigma(d_2)=\sigma(d_2)d_2^{-}d_2$. So, $d_1R \subseteq d_2R$ and $Rd_1 \subseteq Rd_2$.

As $\sigma$ is bijective, then $d_2=\sigma^{-1}(d_1)$, which also
guarantees that $d_2R \subseteq d_1R$ and $Rd_2 \subseteq Rd_1$.

It follows from \cite[Corollary 2.4]{Mary and Patricio} that if $a$
is invertible along $d_1$ with $d_1R=d_2R$ and $Rd_1=Rd_2$, then
$a^{\parallel d_2}$ exists and $a^{\parallel d_1}=a^{\parallel
d_2}$.

So the result follows from Corollary \ref{all}.
\end{proof}

Next, we give two simple examples to illustrate Theorem \ref{lambda}
above.

\begin{example} {\rm (i) Let $R=M_2(\mathbb{C})$ be the ring of all $2 \times 2$ complex
matrices and let $\sigma: R\rightarrow R, \sigma(M)=
\left[\begin{smallmatrix}
              \frac{1}{2}   & 0 \\
              0 & \frac{1}{2}
        \end{smallmatrix}
 \right]M$ for all $M$ in $R$. Then $\sigma$ is a bijective centralizer by Example
 \ref{Ex}(iii). Take $A=\left[\begin{smallmatrix}
              0   & 0 \\
              1 & 1
        \end{smallmatrix}
 \right]$, $B=\left[\begin{smallmatrix}
              1   & 0 \\
              1 & 0
        \end{smallmatrix}
 \right]$, $D_1=\left[\begin{smallmatrix}
              \frac{1}{2}   & \frac{1}{2} \\
              0 & 0
        \end{smallmatrix}
 \right]$ and $D_2=\left[\begin{smallmatrix}
              1   & 1 \\
              0 & 0
        \end{smallmatrix}
 \right]$. We can check $D_1=\sigma(D_2)$, $A^{\parallel D_1}=
 \left[\begin{smallmatrix}
              1   & 1 \\
              0 & 0
        \end{smallmatrix}
 \right]$ and $B^{\parallel D_2}=
 \left[\begin{smallmatrix}
              \frac{1}{2}   & \frac{1}{2} \\
              0 & 0
        \end{smallmatrix}
 \right]$ by direct calculations. Hence $A^{\parallel D_1}+B^{\parallel D_2}=\left[\begin{smallmatrix}
              \frac{3}{2}   & \frac{3}{2} \\
              0 & 0
        \end{smallmatrix}
 \right]=A^{\parallel D_1}(A+B)B^{\parallel D_2}$.

 (ii) Let $R=\mathbb{Z}_9$ and let $\sigma: R\rightarrow R,\sigma(x)= 2x$ for any $x\in R$. It follows from Example
 \ref{Ex}(iii) that $\sigma$ is a bijective centralizer. Take $a=7$, $b=5$, $d_1=4$ and $d_2=2$. Then $d_1=\sigma(d_2)$, $a^{\parallel
d_1}=7^{\parallel 4}=4$ and
 $b^{\parallel d_2}=5^{\parallel 2}=2$. Moreover, $a^{\parallel
d_1}+b^{\parallel d_2}=6=a^{\parallel d_1}(a+b)b^{\parallel d_2}$.}
\end{example}

In \cite{Mary}, one can get $a^{\parallel a}=a^\#$, $a^{\parallel
a^n}=a^D$ and $a^{\parallel a^*}=a^\dag$. By $R^\#$, $R^D$ and
$R^\dag$ we denote the sets of all group, Drazin and Moore-Penrose
invertible elements in $R$, respectively.

As applications of Theorem \ref{lambda}, it follows that

\begin{corollary}  Let $\sigma: R \rightarrow R$ be a bijective centralizer and let $a,b\in R$. Then

\emph{(i)} If $a=\sigma (b)$, then $a^\#+b^\#=a^\#(a+b)b^\#$, for
$a,b\in R^\#$.

\emph{(ii)} If $a^n=\sigma (b^m)$ for some integers $m$ and $n$,
then $a^D+b^D=a^D(a+b)b^D$, for $a,b\in R^D$.

\emph{(iii)} Let $R$ be a ring with involution. If $a^*=\sigma
(b^*)$, then $a^\dag+b^\dag=a^\dag(a+b)b^\dag$, for $a,b\in R^\dag$.

\emph{(iv)} Let $R$ be a ring with involution. If $a=\sigma (b^*)$,
then $a^\#+b^\dag=a^\#(a+b)b^\dag$, for $a\in R^\#$ and $b\in
R^\dag$.
\end{corollary}

\section{Characterizations of the inverse along an element}

We first begin with a lemma on the commutativity of the inverse
along an element.

\begin{lemma} \label{commu} Let $\sigma: R \rightarrow R$ be a bijective centralizer and let $a,d\in R$ with $ad=\sigma(da)$. If $a$ is invertible along $d$,
then $a^{\parallel d}a=aa^{\parallel d}$.
\end{lemma}

\begin{proof} We have $$ad=\sigma(da)=\sigma(a^{\parallel d}ad a)=a^{\parallel d}a\sigma(da)=a^{\parallel d}a^2d. \eqno(3.1)$$  Also, as $\sigma$ is bijective,
then $$da=\sigma^{-1}(ad)=\sigma^{-1}(adaa^{\parallel d})=\sigma^{-1}(\sigma(da)aa^{\parallel d})=da^2a^{\parallel d}.\eqno(3.2)$$

As $a^{\parallel d}\in dR \cap Rd$, then there exist $x,y\in R$ such
that $a^{\parallel d}=xd=dy$. Multiplying the equality (3.1) by $y$
on the right yields $aa^{\parallel d}=ady=a^{\parallel
d}a^2dy=a^{\parallel d}a^2a^{\parallel d}$. Multiplying the equality
(3.2) by $x$ on the left yields $a^{\parallel
d}a=xda=xda^2a^{\parallel d}=a^{\parallel d}a^2a^{\parallel d}$.

Hence, $a^{\parallel d}a=aa^{\parallel d}$.
\end{proof}

In \cite[page 7]{Zhu reverse}, the authors gave a counterexample to
illustrate that reverse order law for the inverse along an element
does not hold in general.

The following theorem, extending \cite[Theorem 2.14]{Zhu reverse},
considers the reverse order law for the inverse along an element,
under a generalized commutativity condition.

\begin{theorem} \label{reverse lambda} Let $\sigma: R \rightarrow R$ be a bijective centralizer and let $a,b,d\in R$ with $ad=\sigma(da)$. If $a^{\parallel d}$
and $b^{\parallel d}$ exist, then

\emph{(i)} $(ab)^{\parallel d}$ exists and $(ab)^{\parallel
d}=b^{\parallel d} a^{\parallel d}$.

\emph{(ii)} $(ba)^{\parallel d}$ exists and $(ba)^{\parallel
d}=a^{\parallel d} b^{\parallel d}$.
\end{theorem}

\begin{proof}  Applying Lemma \ref{absorption law}(ii), it follows $b^{\parallel d}aa^{\parallel d}=b^{\parallel d}$. Also, $a^{\parallel d}a=aa^{\parallel d}$
by Lemma \ref{commu}. Then $b^{\parallel d}a^{\parallel d}abd=b^{\parallel d}aa^{\parallel d}bd=b^{\parallel d}bd=d$ and
$dabb^{\parallel d} a^{\parallel d}=\sigma^{-1}(ad)bb^{\parallel d} a^{\parallel d}=\sigma^{-1}(a)(dbb^{\parallel d})a^{\parallel d}=\sigma^{-1}(ad)a^{\parallel d}
=daa^{\parallel d}=d$. Finally, $b^{\parallel d} a^{\parallel d} \in dR \cap Rd$ since $a^{\parallel d}\in dR \cap Rd$ and $b^{\parallel d}\in dR \cap Rd$.

Hence, $ab$ is invertible along $d$ and $(ab)^{\parallel
d}=b^{\parallel d} a^{\parallel d}$.

(ii) The proof is similar to (i).
\end{proof}

\begin{remark} {\rm According to proof of Theorem \ref{reverse lambda}, one can see that it indeed holds in a semigroup.}
\end{remark}

\begin{corollary} \label{1} Let $a,b,d \in R$ with $ad=da$. If $a^{\parallel d}$ and $b^{\parallel d}$ exist, then

\emph{(i)} $(ab)^{\parallel d}$ exists and $(ab)^{\parallel
d}=b^{\parallel d} a^{\parallel d}$.

\emph{(ii)} $(ba)^{\parallel d}$ exists and $(ba)^{\parallel
d}=a^{\parallel d} b^{\parallel d}$.
\end{corollary}

Let $R=\mathbb{Z}_7$. Then 5 is invertible along 3 and $5^{\parallel
3}=3$. We notice the following fact, i.e., $5$ is also invertible
along $ 2\cdot 3$ and $5^{\parallel {2 \cdot 3}}=3=5^{\parallel 3}$,
and $2 \cdot 5$ is also invertible along $3$ and $(2 \cdot
5)^{\parallel 3}=5=4 \cdot 3=2^{-1} \cdot 5^{\parallel 3}$.
Motivated by this, we get the following proposition.

\begin{proposition} \label{sigma} Let $\sigma: R \rightarrow R$ be a bijective centralizer and let $a,d\in R$. Then

\emph{(i)} $a^{\parallel d}$ exists if and only if $a^{\parallel
\sigma (d)}$ exists. Moreover, $a^{\parallel \sigma
(d)}=a^{\parallel d}$.

\emph{(ii)} $a^{\parallel d}$ exists if and only if $(\sigma
(a))^{\parallel d}$ exists. Moreover, $(\sigma (a))^{\parallel d}=
\sigma^{-1} (a^{\parallel d})$.
\end{proposition}

\begin{proof} (i) ``$\Rightarrow$'' As $a^{\parallel d}$ exists, then $d$ is regular and hence $\sigma(d)=dd^{-}\sigma(d)=\sigma(d)d^{-}d$. So,
 $\sigma(d)R \subset dR$ and $R\sigma(d)\subset Rd$. Also, $\sigma$ is bijective, then $d=\sigma(d)\sigma^{-1}(1)=\sigma^{-1}(1)\sigma(d)$ implies
  $dR\subset \sigma(d)R$ and $Rd\subset R\sigma(d)$. Hence, $dR= \sigma(d)R$ and $Rd= R\sigma(d)$. It follows from \cite[Corollary 2.4]{Mary and Patricio}
  that $a^{\parallel \sigma (d)}$ exists and $a^{\parallel \sigma (d)}=a^{\parallel d}$.

``$\Leftarrow$'' By noting that $\sigma^{-1}$ is a bijective
centralizer.

(ii) One can check it directly.
\end{proof}

We next give an existence criterion of the inverse along an element
by centralizers in a ring. Herein, a lemma is presented.

\begin{lemma} \label{Jacobson} Let $a, b, c\in R$. Then

\emph{(i)} If $(1+ab)c = 1$, then $(1+ba)(1-bca) = 1$.

\emph{(ii)} If $c(1+ab) = 1$, then $(1-bca)(1+ba) = 1$.
\end{lemma}

It follows from Lemma \ref{Jacobson} that $1+ab$ is left (right)
invertible if and only if $1+ba$ is left (right) invertible. In
particular,  $1+ab$ is invertible if and only if $1+ba$ is
invertible. Moreover, $(1+ba)^{-1} = 1-b(1+ab)^{-1}a$. This formula
is known as Jacobson's Lemma.

Let $\sigma: R\rightarrow R$ be a bijective centralizer and let $a
\in R$. It is known that if $a$ is regular then so is $\sigma(a)$
and $(\sigma(a))^{-}=\sigma^{-1} (a^{-})$.

\begin{theorem} \label{gmary} Let $a,d\in R$ with $d$ regular. If $\sigma: R \rightarrow R$ is a bijective centralizer, then the following conditions are
equivalent{\rm :}

\emph{(i)} $a^{\parallel d}$ exists.

\emph{(ii)} $u=\sigma (da)+1-dd^{-}$ is invertible.

\emph{(iii)} $v=\sigma (ad)+1-d^{-}d$ is invertible.

In this case, $a^{\parallel d}=\sigma(u^{-1})d=d\sigma(v^{-1})$.
\end{theorem}

\begin{proof} (ii) $\Leftrightarrow$ (iii) As $\sigma$ is a centralizer, then $\sigma(da)=d\sigma(a)$ and $\sigma(ad)=\sigma(a)d$.
Take $x=-d$ and $y=\sigma(a)-d^{-}$. Then,
$u=d\sigma(a)+1-dd^{-}=1-xy$ is invertible if and only if
$1-yx=\sigma(a)d+1-d^{-}d=v$ is invertible by Jacobson's Lemma.

(i) $\Rightarrow$ (ii) Suppose that $a^{\parallel d}$ exists. Then
$a^{\parallel \sigma (d)}$ exists by Proposition \ref{sigma}. Since
$a^{\parallel \sigma (d)}\in \sigma (d)R$, there exists $x\in R$
such that $a^{\parallel \sigma (d)}=\sigma (d)x$ and hence $\sigma
(d)=\sigma (d) a a^{\parallel \sigma (d)}=\sigma (d a) \sigma (d)x$,
which implies $d=\sigma(da)dx$. As
$(\sigma(da)dd^{-}+1-dd^{-})(dxd^{-}+1-dd^{-})=1$, then $\sigma
(da)dd^{-}+1-dd^{-}$ is right invertible. Lemma \ref{Jacobson}
ensures that $dd^{-}\sigma (da)+1-dd^{-}$, i.e., $\sigma
(da)+1-dd^{-}$ is  right invertible.

Also as $a^{\parallel \sigma (d)}\in R\sigma (d)$, then there exists
$y\in R$ such that $a^{\parallel \sigma (d)}=y\sigma(d)$. So
$\sigma(d)=y \sigma(d)a\sigma(d)$. Note that
$(d^{-}yd+1-d^{-}d)(d^{-}d\sigma(ad)+1-d^{-}d)=1$ implies that
$d^{-}d\sigma(ad)+1-d^{-}d$ is left invertible. It follows from
Lemma \ref{Jacobson} that $\sigma(ad)+1-d^{-}d$ is left invertible.
 Again, by Lemma \ref{Jacobson}, we get $\sigma (da)+1-dd^{-}$ is left invertible.

Therefore, $u=
\sigma (da)+1-dd^{-}$ is invertible.

(ii) $\Rightarrow$ (i) As $u$ is invertible, then $v$ is also
invertible. From $ud=\sigma(da)d=d\sigma(ad)=dv$, we get
$d=u^{-1}\sigma(dad)=\sigma(u^{-1})dad=dad\sigma(v^{-1})$. Hence,
$a^{\parallel d}$ exists in terms of \cite[Theorem 2.2]{Mary and
Patricio}.

Now, we show that $m=\sigma(u^{-1})d=d\sigma(v^{-1})$ is the inverse
of $a$ along $d$. As $u^{-1}d=dv^{-1}$, then
$\sigma(u^{-1})d=d\sigma (v^{-1})$. It follows that
$mad=\sigma(u^{-1})dad=u^{-1}\sigma(dad)=d=dad\sigma(v^{-1})=dam$
and $m\in dR\cap Rd$.

Thus, $a^{\parallel d}=\sigma(u^{-1})d=d\sigma(v^{-1})$.
\end{proof}

\begin{remark} {\rm The assumption `` a bijective centralizer" in Theorem
\ref{gmary} above cannot be replaced by ``a centralizer''. Such as
let $R=\mathbb{Z}_6$ and let $\sigma: R\rightarrow R, \sigma(x)= 3x$
for all $x\in R$. Then $\sigma$ is a centralizer but not bijective.
Take $a=4$ and $d=2$ in $R$. Then $a$ is invertible along $d$ and
$a^{\parallel d}=4$. However, $\sigma(da)+1-dd^{-}=3\cdot 2+1-2
\cdot 2=3$ is not invertible.}
\end{remark}

\begin{corollary} {\rm \cite[Theorem 3.2]{Mary and Patricio}} \label{sigma1} Let $a,d\in R$ with $d$ regular.
Then the following conditions are equivalent{\rm :}

\emph{(i)} $a^{\parallel d}$ exists.

\emph{(ii)} $u=da+1-dd^{-}$ is invertible.

\emph{(iii)} $v=ad+1-d^{-}d$ is invertible.

In this case, $a^{\parallel d}=u^{-1}d=dv^{-1}$.
\end{corollary}

It is known \cite[Theorem 11(ii)]{Mary} that $a\in R^D$ implies that
$a^{\parallel a^n}$ exists for some integer $n$. Note that
$a^{\parallel d}$ exists implies that $d$ is regular from \cite{Mary
and Patricio}. So, applying Theorem \ref{gmary}, we have the
following result which is a new characterization of Drazin inverse
of elements in a ring.

\begin{corollary} \label{Centralizer Drazin inverse} Let $\sigma: R \rightarrow R$ is a bijective centralizer and let $a\in R$.
Then the following conditions are equivalent{\rm:}

\emph{(i)} $a\in R^D$.

\emph{(ii)} $a^n$ is regular and $u=\sigma(a^{n+1})+1-a^n(a^n)^{-}$
is invertible, for some integer $n$.

\emph{(iii)} $a^n$ is regular and $v=\sigma(a^{n+1})+1-(a^n)^{-}a^n$
is invertible, for some integer $n$.

In this case, $a^D=\sigma(u^{-1})a^n=a^n \sigma(v^{-1})$.
\end{corollary}

Taking $d=a$ in Theorem \ref{gmary} above, we get an existence
criterion of the group inverse of $a$ in a ring, extending
\cite[Proposition 2.1]{Patricio2009}, i.e., $a\in R^\#$ exists if
and only if $a^2+1-aa^{-}$ is invertible if and only if
$a^2+1-a^{-}a$ is invertible, for a regular element $a$.

\begin{corollary} Let $a\in R$ be regular. If $\sigma: R \rightarrow R$ is a bijective centralizer, then the following conditions are equivalent{\rm:}

\emph{(i)} $a\in R^\#$.

\emph{(ii)} $u=\sigma(a^2)+1-aa^{-}$ is invertible.

\emph{(iii)} $v=\sigma(a^2)+1-a^{-}a$ is invertible.

In this case, $a^\#=\sigma(u^{-1})a=a\sigma(v^{-1})$.
\end{corollary}

Next, we consider the characterizations of the Moore-Penrose inverse
of a regular element by centralizers in a ring.

It follows from from \cite[Theorem 2.16]{Zhu and chen} that $a\in
R^\dag\Leftrightarrow a\in aa^*aR \Leftrightarrow a\in Raa^*a$,
which allow us to characterize Moore-Penrose inverses in terms of
one-sided invertibilities in a ring. Also, we know \cite[Theorems
2.19 and 2.20]{Zhu and chen} that $a=aa^*ax$ or $a=yaa^*a$ implies
that $a$ is Moore-Penrose invertible and
$a^\dag=a^*ax^2a^*=a^*y^2aa^*$.

\begin{theorem} \label{onesided MP} Let $R$ be a ring with involution and let $a\in R$ be regular.
If $\sigma: R \rightarrow R$ is a bijective centralizer, then the following conditions are equivalent{\rm:}

\emph{(i)} $a\in R^\dag$.

\emph{(ii)} $u=\sigma(aa^*)+1-aa^{-}$ is right (left) invertible.

\emph{(iii)} $v=\sigma(a^*a)+1-a^{-}a$ is right (left) invertible.

In this case,
$a^\dag=a^*(\sigma(u_l^{-1}))^2aa^*=a^*a(\sigma(v_r^{-1}))^2a^*$,
where $u_l^{-1}$, $v_r^{-1}$ denote a left inverse of $u$ and a
right inverse of $v$, respectively.
\end{theorem}

\begin{proof} For simplicity, we just prove the case of right invertibility.

(ii) $\Leftrightarrow$ (iii) As $\sigma$ is a centralizer, then
$\sigma(aa^*)=a\sigma(a^*)$ and $\sigma(a^*a)=\sigma(a^*)a$. Pose
$x=-a$ and $y=\sigma(a^*)-a^{-}$. Then $1-xy=u$ is right invertible
if and only if $1-yx=v$ is right invertible by Lemma \ref{Jacobson}.

(i) $\Rightarrow$ (ii) Since $a\in R^\dag$, there exists $x\in R$
such that $a=aa^*ax$ from \cite[Theorem 2.16]{Zhu and chen}. As
\begin{eqnarray*}
&&(\sigma(aa^*)aa^{-}+1-aa^{-})(a\sigma^{-1}(x)a^{-}+1-aa^{-})\\
&=& \sigma(aa^*)a\sigma^{-1}(x)a^{-}+1-aa^{-}\\
&=&\sigma(aa^*a)\sigma^{-1}(x)a^{-}+1-aa^{-}\\
&=&aa^*axa^{-}+1-aa^{-}\\
&=&aa^{-}+1-aa^{-}\\
&=&1,
\end{eqnarray*}
then $\sigma(aa^*)aa^{-}+1-aa^{-}$ is right invertible. Again, Lemma
\ref{Jacobson} guarantees that $u=\sigma(aa^*)+1-aa^{-}$ is right
invertible.

(ii) $\Rightarrow$ (i) Note that $u$ and hence $v$ are both right
invertible. We have $av=\sigma(aa^*a)$. Hence,
$a=\sigma(aa^*a)v_r^{-1}=aa^*a\sigma(v_r^{-1})\in aa^*aR$, which
means $a\in R^\dag$ by \cite[Theorem 2.16]{Zhu and chen}.

As $ua=\sigma(aa^*a)$, then
$a=u_l^{-1}\sigma(aa^*a)=\sigma(u_l^{-1})aa^*a$. Hence,
$a^\dag=a^*(\sigma(u_l^{-1}))^2aa^*$ in terms of \cite[Theorem
2.20]{Zhu and chen}.

Similarly, $a=aa^*a\sigma(v_r^{-1})$ ensures that
$a^\dag=a^*a(\sigma(v_r^{-1}))^2a^*$ in terms of \cite[Theorem
2.19]{Zhu and chen}.
\end{proof}

\begin{remark} {\rm The expressions for the Moore-Penrose inverse $a^\dag$ in Theorem \ref{onesided MP} can also be given by $u_r^{-1}$ or $v_l^{-1}$.
Indeed, by Lemma \ref{Jacobson}, we have
$v_r^{-1}=1-(\sigma(a^*)-a^{-})u_r^{-1}a$. If we replace $v_r^{-1}$
by $1-(\sigma(a^*)-a^{-})u_r^{-1}a$ in the equality
$a^\dag=a^*a(v_r^{-1})^2a^*$. then the formula of $a^\dag$ can be
presented by $u_r^{-1}$. The similar process for $u_l^{-1}$ follows
the formula of $a^\dag$ by $v_l^{-1}$.}
\end{remark}

\begin{corollary} \label{Cor of onesided MP} Let $R$ be a ring with involution and let $a\in R$ be regular. Then the following conditions are equivalent{\rm:}

\emph{(i)} $a\in R^\dag$.

\emph{(ii)} $u=aa^*+1-aa^{-}$ is right (left) invertible.

\emph{(iii)} $v=a^*a+1-a^{-}a$ is right (left) invertible.

In this case, $a^\dag=a^*(u_l^{-1})^2aa^*=a^*a(v_r^{-1})^2a^*$,
where $u_l^{-1}$, $v_r^{-1}$ denote a left inverse of $u$ and a
right inverse of $v$, respectively.
\end{corollary}

It follows from \cite[Theorem 3.3]{Zhu reverse} that another
Moore-Penrose inverse of $a$ in Corollary \ref{Cor of onesided MP}
can also be given as
$a^\dag=(u_l^{-1}a)^*a(u_l^{-1}a)^*=(av_r^{-1})^*a(av_r^{-1})^*$.

Given a regular element $a\in R$, it follows from Theorem
\ref{onesided MP} that $a\in R^\dag$ if and only if $u$ (resp., $v$)
is right invertible if and only if $u$ (resp., $v$) is left
invertible. Hence, we have the following corollary.

\begin{corollary} \label{MP} Let $R$ be a ring with involution and let $a\in R$ be regular. If $\sigma: R \rightarrow R$ is a bijective centralizer,
then the following conditions are equivalent{\rm:}

\emph{(i)} $a\in R^\dag$.

\emph{(ii)} $u=\sigma(aa^*)+1-aa^{-}$ is invertible.

\emph{(iii)} $v=\sigma(a^*a)+1-a^{-}a$ is invertible.

In this case,
$a^\dag=a^*(\sigma(u^{-1}))^2aa^*=a^*a(\sigma(v^{-1}))^2a$.
\end{corollary}

\begin{remark} {\rm In Corollary \ref{MP} above, the expression of $a^\dag$ can also be given by $\sigma(u^{-1})a^*=a^*\sigma(v^{-1})$ by Theorem \ref{gmary}.}
\end{remark}

Setting the centralizer $\sigma=1$ in Corollary \ref{MP}, it follows
that

\begin{corollary} {\rm \cite[Theorem 2.1]{Patricio and Mendes}} Let $R$ be a ring with involution and let $a\in R$ be regular. Then the following conditions
 are equivalent{\rm :}

\emph{(i)} $a\in R^\dag$.

\emph{(ii)} $u=aa^*+1-aa^{-}$ is invertible.

\emph{(iii)} $v=a^*a+1-a^{-}a$ is invertible.

In this case, $a^\dag=a^*(u^{-1})^2aa^*=a^*a(v^{-1})^2a$.
\end{corollary}

\centerline {\bf ACKNOWLEDGMENTS} This research was carried out by
the first author during his visit to the Department of Mathematics
and Applications, University of Minho, Portugal. He gratefully
acknowledges the financial support of China Scholarship Council.
This research is also supported by the National Natural Science
Foundation of China (No. 11371089), the Natural Science Foundation
of Jiangsu Province (No. BK20141327), the Scientific  Innovation
Research  of College Graduates in Jiangsu Province (No. CXLX13-072),
the Scientific Research Foundation of Graduate School of Southeast
University, the FEDER Funds through ¡®Programa Operacional Factores
de Competitividade-COMPETE', the Portuguese Funds through FCT-
`Funda\c{c}\~{a}o para a Ci\^{e}ncia e a Tecnologia', within the
project UID-MAT-00013/2013.

\end{document}